\documentclass[12pt]{article}
\usepackage{amsfonts,amsthm, latexsym, amssymb, mathrsfs, graphicx, amsmath}
\usepackage{mathtools}
\usepackage[dvipsnames]{xcolor}
\usepackage{bm, wasysym, tikz}
\usetikzlibrary{arrows, shapes}
\tikzset{>=stealth'}
\newcommand{\hes}{{\rm Hess\,}}

\begingroup
\newtheorem{theorem}{Theorem}[section]
\newtheorem*{theorem*}{Theorem}
\newtheorem{lemma}[theorem]{Lemma}
\newtheorem{proposition}[theorem]{Proposition}

\endgroup

\title{\large{\bf An antimaximum principle for periodic solutions of a forced oscillator}}

\author{}

\date{}

\begin{document}
	
	\maketitle
	
	\begin{center}
	{\bf Alain Albouy$^{1}$,\qquad Antonio J. Ure\~na$^{2}$}
	
	\bigskip
	$^{1}$ IMCCE, CNRS-UMR8028, Observatoire de Paris
	
        77, avenue Denfert-Rochereau, 75014 Paris, France
         
         Alain.Albouy@obspm.fr
         
         \bigskip
         
         $^{2}$ Departamento de Matematica Aplicada, Facultad de Ciencias
         
         Campus Universitario de Fuentenueva
         
         Universidad de Granada, 18071, Granada, Spain
         
        ajurena@ugr.es
        
        \bigskip
	\end{center}

{\bf Abstract.} Consider the equation of the linear oscillator $u''+u=h(\theta)$, where the forcing term $h:\mathbb R\to\mathbb R$ is $2\pi$-periodic and positive. We show that the existence of a periodic solution implies the existence of a positive solution. To this aim we establish connections between this problem and some separation questions of convex analysis.

\medskip

\section{Introduction}
\label{sec:intro}
Consider the linear, second-order equation

\begin{equation}\label{eq1}
u''+\omega^2 u=h(\theta)\,,
\end{equation}
where $\omega>0$ is a positive constant, the prime symbol stands for differentiation with respect to the independent variable $\theta$, the function $h:\mathbb R\to\mathbb R$ is continuous and $2\pi$-periodic, and the twice continuously differentiable function $u=u(\theta)$ is the unknown. It models the simplest oscillations that one can think of: those of a forced harmonic spring in one spatial dimension. The well-known Fredholm alternative theorem (see, e.g., \cite[Lemma 1.1, p. 146]{Hal}) provides a precise answer to the question of existence and multiplicity of $2\pi$-periodic solutions. More precisely, if $\omega$ is not an integer then \eqref{eq1} has a unique $2\pi$-periodic solution, and if $\omega\in\mathbb Z$ then there are periodic solutions if and only if the nonresonance condition 
\begin{equation}\label{eq2}
\int_0^{2\pi}h(\theta)\cos\omega\theta\,d\theta=\int_0^{2\pi}h(\theta)\sin\omega\theta\,d\theta=0
\end{equation}
holds. Morever, in this situation all solutions are $2\pi$-periodic. Assuming now that $h=h(\theta)$ has a sign we consider the question: 
\begin{center}
{\em Does \eqref{eq2} imply the existence of solutions $u$ of \eqref{eq1} having the sign of} $h$?
\end{center}
Under the label of {\em antimaximum principles}, this problem has been studied in depth when $\omega\not\in\mathbb Z$ . For instance, it follows from \cite[Theorem 1.1]{Zhang} that the answer is positive for $0<\omega\leq 1/2$ and negative for  $1/2<\omega\not\in\mathbb Z$.
The main result of this note is the following:
\begin{theorem}\label{th}
{Assume that $\omega=1$. If $h=h(\theta)$ satisfies \eqref{eq2}, is nonnegative, and not identically zero, then \eqref{eq1} has a solution $u=u(\theta)$ with $u(\theta)>0$ for all $\theta\in\mathbb R$.}
\end{theorem}
Some remarks are in order:

\medskip

\noindent {\em (i)} Equivalently we may restate this theorem as follows: in the case $\omega=1$, if $0\not\equiv h\in C(\mathbb R/2\pi\mathbb Z,\mathbb R)$ is nonnegative and equation \eqref{eq1} has a periodic solution, then it has a  everywhere-positive solution.

\medskip

\noindent {\em (ii)} One may wonder whether Theorem \ref{th} remains true for $\omega=2,3,4...$ In Proposition \ref{sec4} we shall provide a counterexample showing that the answer is negative for $\omega\geq 3$. The case $\omega=2$ is not well understood by the authors.

\medskip

\noindent {\em (iii)} We point out that \eqref{eq1} can be solved in a number of elementary ways that are found in classical textbooks.  For instance, for $\omega=1$ the formula of variation of constants gives $$u(\theta)=\left(\alpha+\int_0^\theta h(\tau)\cos\tau\,d\tau\right)\sin\theta-\left(\beta+\int_0^\theta h(\tau)\sin\tau\,d\tau\right)\cos\theta$$ with $\alpha,\beta\in\mathbb R$. Alternatively one can use Fourier series: if  $h(\theta)=a_0+\sum_{n=2}^{+\infty}(a_n\cos n\theta+b_n\sin n\theta)$ then the solutions of \eqref{eq1} are given by
$$u(\theta)=a_0-\sum_{n=2}^{+\infty}\left(\frac{a_n}{n^2-1}\cos n\theta+\frac{b_n}{n^2-1}\sin n\theta\right)+\alpha\cos\theta+\beta\sin\theta,$$
with $\alpha,\beta\in\mathbb R$. Analyzing the sign of these expressions seems hard and so both approaches are apparently ill-suited to answer the question formulated above.

\medbreak

\noindent {\em (iv)} A more sophisticated study would involve the method of lower and upper solutions, for which we refer to the classical work \cite{deC-Hab}. Indeed, for nonnegative $h$ equation \eqref{eq1} has the upper solution $\beta(\theta)\equiv 0$ and the lower solution $\alpha(\theta)\equiv M$ with $M>0$ big. Notice however that they are in the reverse order. Even though a number of results of the lower and upper solutions literature consider the reverse order case (see, e.g., \cite{Zhang}), we are unaware of any that might apply to \eqref{eq1}.

\medbreak

Our approach to the problem will be quite different and, to the best of our knowledge, new. In broad terms, the key idea  will consist in relating the different periodic solutions of equation \eqref{eq1} with the supporting linear forms of a given convex, positively homogeneous map of degree $1$.

\section{Convexity and supporting linear forms}
The following terminology is usual. A function $\rho:\mathbb R^N\to\mathbb R$ is said to be positively homogeneous of degree $1$ provided that
$$\rho(\lambda z)=\lambda\rho(z),\qquad z\in\mathbb R^N,\ \lambda> 0.$$
Assume now that $\rho:\mathbb R^N\to\mathbb R$ is positively homogeneous of degree $1$ and satisfies $\rho(0)=0$. Given $e^*\in(\mathbb R^{N})^*$,
\begin{itemize}
	\item it is usually said that $e^*$ is a  {\em supporting linear form} for $\rho$ if
	$$\rho(z)\geq\langle z,e^*\rangle\text{ for every }z\in\mathbb R^N;$$
	\item we shall say that $e^*$  is a {\em strictly supporting linear form} for $\rho$ provided that
	\begin{equation}\label{euu1}
	\rho(z)>\langle z,e^*\rangle\text{ for every }z\in\mathbb R^N\backslash\{0\}.
	\end{equation}  
\end{itemize}

On the other hand, one easily checks that if $\rho:\mathbb R^N\to\mathbb R$ is positively homogeneous of degree $1$, then it is also convex if and only if 
$$\rho(z_1+z_2)\leq\rho(z_1)+\rho(z_2),\qquad z_1,z_2\in\mathbb R^N.$$
Since every convex function defined on (an open convex subset of) a finite-dimensional vector space is continuous (see e.g. \cite[Theorem 5.2.1]{florlev}), in this situation we recover the previous condition  $\rho(0)=0$. 

\medskip

The classical version of the Hahn-Banach theorem (see, e.g. \cite[p.1, Th\'eor\`eme I.1]{Bre}) states that every function $\rho:\mathbb R^N\to\mathbb R$ which is  positively homogeneous of degree $1$ and convex  admits supporting linear forms. The main result of this section is the following
\begin{proposition}\label{prop1}Let  $\rho:\mathbb R^N\to\mathbb R$ be positively homogeneous of degree $1$, convex and differentiable at every point of $\mathbb R^N\backslash\{0\}$. If in addition $\rho\not\in(\mathbb R^N)^*$, then it admits a strictly supporting linear form.
	\end{proposition}  
This variation of the Hahn-Banach theorem can be deduced from known results (see \cite[(9.12), p. 459]{klee}), but we shall give a hopefully simpler proof. Our approach will be divided in two steps. In the first one we shall see that under the assumptions of Proposition \ref{prop1} one has the inequality
	\begin{equation}\label{eq6}
	\rho(z)+\rho(-z)>0\text{ for every }z\in\mathbb R^N\backslash\{0\}.
\end{equation}
This is indeed a first step towards the proof of Proposition \ref{prop1} since condition \eqref{eq6} is clearly necessary for the existence of strictly supporting linear forms. 
\begin{lemma}\label{lem2}{Let $\rho:\mathbb R^N\to\mathbb R$ be positively homogeneous of degree $1$ and convex. If moreover $\rho\not\in(\mathbb R^N)^*$, then $\rho(z)+\rho(-z)>0$ for every $z\in\mathbb R^N\backslash\{0\}$ where $\rho$ is differentiable.
	
	}\begin{proof}Convexity gives
		$$0=\rho(0)\leq\frac{\rho(z)+\rho(-z)}{2},$$
		for every $z\in\mathbb R^N$. Thus, it will suffice to check that $\rho(z)+\rho(-z)\not=0$ for every $z\in\mathbb R^N\backslash\{0\}$ where $\rho$ is differentiable.	We use an inductive argument on the dimension $N$. 
		
		\medskip
		
		If $N=1$ the result is trivial: if $\rho:\mathbb R\to\mathbb R$ is positively homogeneous of degree $1$, the equality $\rho(-z_0)+\rho(z_0)=0$ for some $z_0\not=0$ implies that $\rho$ is linear.  
		
		\medskip	
		
		Assuming the statement true for some $N$, let  $\rho:\mathbb R^{N+1}\to\mathbb R$ be positively homogeneous of degree $1$, convex, and not linear, and suppose, by a contradiction argument, that $\rho(z_0)+\rho(-z_0)=0$ for some $z_0\in\mathbb R^{N+1}\backslash\{0\}$ where $\rho$ is differentiable. Choose some  hyperplane $H\subset\mathbb R^{N+1}$ with $z_0\in H$. By the inductive assumption the restriction $\rho\big|_H=e^*$ must be linear. Fix some vector $z_1\in\mathbb R^{N+1}\backslash H$ and let the function $\rho_1:H\to\mathbb R$ be defined by
		$$\rho_1(h):=\rho(z_1+h)-\langle e^*,h\rangle,\qquad h\in H.$$
		Then $\rho_1$ is convex. Moreover, denoting by $|\cdot|$ a fixed (arbitrary) norm in $\mathbb R^N$ one has
	\begin{multline*}
		\lim_{|h|\to\infty}\frac{\rho_1(h)}{|h|}=\lim_{|h|\to\infty}\left(\rho\left(\frac{z_1}{|h|}+\frac{h}{|h|}\right)-\left\langle e^*,\frac{h}{|h|}\right\rangle\right)=\\	=\lim_{|h|\to\infty}\left(\rho\left(\frac{z_1}{|h|}+\frac{h}{|h|}\right)-\rho\left(\frac{h}{|h|}\right)\right)=0,
		\end{multline*}
since  $\rho$ is continuous, thus uniformly continuous on compact sets, and $z_1/|h|\to 0$. Hence, we see that $\rho_1:H\to\mathbb R$ is convex and has sublinear growth at infinity, and we deduce that $\rho_1(h)\equiv c$ is constant. Then,
		$$\rho(h+z_1)=\langle e^*,h\rangle+c,\qquad h\in H,$$
		and, by homogeneity,
		$$\rho(h+\lambda z_1)=\langle e^*,h\rangle+c\lambda,\qquad h\in H,\ \lambda\geq 0.$$
		Similarly, there exists another constant $d\in\mathbb R$ such that $\rho(h+\lambda z_1)=\langle e^*,h\rangle+d\lambda$ for every $h\in H$ and $\lambda\leq 0$. In particular, taking $h=z_0\in H$ we see that $$\rho(z_0+\lambda z_1)=\begin{cases}\langle e^*,z_0\rangle+c\lambda&\text{if }\lambda\geq 0,\\
		\langle e^*,z_0\rangle+d\lambda&\text{if }\lambda<0.\end{cases}$$ Remembering that $\rho$ was assumed differentiable at $z_0$ we deduce that $c=d$. It means that $\rho$ is linear, a contradiction. The result follows.
	\end{proof}
	
\end{lemma}
To complete the proof of Proposition \ref{prop1} we shall check the following lemma. It goes back to Minkowski \cite[XXV, pp. 151-153]{Min}, but we include a short proof for completeness. 

\begin{lemma}\label{lem3}{Let $\rho:\mathbb R^N\to\mathbb R$ be positively homogeneous of degree $1$, convex and satisfy \eqref{eq6}. Then it admits a strictly supporting linear form.
	}
	\begin{proof}We argue by induction on the dimension $N$. If $N=1$ then the result is trivial. Assuming that it holds for some $N$, let $\rho:\mathbb R^{N+1}\to\mathbb R$ satisfy the assumptions above. The identification $\mathbb R^{N+1}\equiv\mathbb R^N\times\mathbb R$ allows us to write the points of $\mathbb R^{N+1}$ as pairs $(z,\lambda)$ where $z\in\mathbb R^N$ and $\lambda\in\mathbb R$. Since the function $\rho_0:\mathbb R^N\to\mathbb R$ given by $z\mapsto\rho(z,0)$ is positively homogeneous of degree $1$, convex, and satisfies \eqref{eq6}, then it admits a strictly supporting linear form $e^*_0\in(\mathbb R^N)^*$. Let $e^*\in(\mathbb R^{N+1})^*$ be defined by $$\langle e^*,(z,\lambda)\rangle:=\langle e^*_0,z\rangle+\alpha\lambda\text{ for every }(z,\lambda)\in\mathbb R^{N+1},$$ where $\alpha$ is a real number, to be fixed later. Homogeneity means that $e^*$ is a strictly supporting linear form for $\rho$ if and only if the inequality in \eqref{euu1} holds for every ordered pair of the form $(z,\pm 1)$ with $z\in\mathbb R^N$. Equivalently, if and only if
		\begin{equation}\label{eu6}
			\rho_1(z):=\langle e_0^*,z\rangle-\rho(z,-1)<\alpha<\rho_2(z):=\rho(z,1)-\langle e_0^*,z\rangle,\qquad z\in\mathbb R^N.
		\end{equation}
		All three functions $\rho,\rho_1,\rho_2$ are continuous. On the other hand homogeneity gives
		$$\limsup_{|z|\to+\infty}\frac{\rho_1(z)}{|z|}=\max_{|z|=1}\Big(\langle e_0^*,z\rangle-\rho_0(z)\Big)<0,$$
		and therefore, $\rho_1(z)\to-\infty$ as $|z|\to\infty$. Similarly, $\rho_2(z)\to+\infty$ as $|z|\to\infty$ and we see that $\rho_1$ attains its global maximum and $\rho_2$ attains its global minimum. Therefore, it will be possible to find a constant $\alpha$ satisfying \eqref{eu6} if and only if $\rho_1(z_1)<\rho_2(z_2)$ for every $z_1,z_2\in\mathbb R^N$. Equivalently,
		$$\rho_2(z_2)-\rho_1(z_1)=\rho(z_2,1)+\rho(z_1,-1)-\langle e_0^*,z_1+z_2\rangle>0,\qquad z_1,z_2\in\mathbb R^N.$$
		
	When $z_2=-z_1$ we obtain the inequality $\rho(z_2,1)>\rho(-z_2,-1)$, which holds true  by assumption \eqref{eq6}.  When $z_2\not=-z_1$ one can use the inductive assumption to obtain $\rho(z_2,1)+\rho(z_1,-1)\geq\rho(z_1+z_2,0)>\langle e_0^*,z_1+z_2\rangle$. It proves the result.
	\end{proof}
\end{lemma}     

Proposition \ref{prop1} will be needed in the next section in the case $N=2$.  

   \section{Periodic functions of one variable  vs. positively homogeneous functions on the plane}  
Let the linear spaces $\mathcal H$, $\mathcal F$ be defined by:
$$\mathcal H:=\big\{\rho\in C(\mathbb R^2,\mathbb R)\cap C^2(\mathbb R^2\backslash\{0\},\mathbb R):\rho\text{ is positively homogeneous of degree }1\big\},$$
$$\mathcal F:=C^2(\mathbb R/2\pi\mathbb Z,\mathbb R).$$

One can construct a linear map $\Phi:\mathcal H\to\mathcal F$ in the following way. Given $\rho\in\mathcal H$ we set $\Phi[\rho]:=u$, where $$u(\theta)=\rho(\cos\theta,\sin\theta),\qquad \theta\in\mathbb R.$$ It is clear that $\Phi$ is bijective: for each $u\in\mathcal F$ there exists an unique $\rho\in\mathcal H$ such that $\Phi[\rho]=u$. In fact, in polar coordinates
$\rho=\Phi^{-1}[u]$ is explicitly given by
\begin{equation}\label{eq3}
	\begin{cases}
	\rho(re^{i\theta})=ru(\theta), &r\geq 0,\ \theta\in\mathbb R/2\pi\mathbb Z,\\	
\rho(0)=0,	
\end{cases}
\end{equation}
where we use complex notation and write $re^{i\theta}:=(r\cos\theta, r\sin\theta)$. In the result below we point out another property of this correspondence.
\begin{lemma}\label{lem1}{Let $\rho\in\mathcal H$ be given and set $u:=\Phi[\rho]$. Then $\rho$ is convex if and only if $u''(\theta)+u(\theta)\geq 0$ for every $\theta\in\mathbb R$.}
\begin{proof}
	Differentiation in \eqref{eq3} shows that the matrix $M$ of $\hes\rho(re^{i\theta})$ with respect to the orthonormal basis $\{e^{i\theta},ie^{i\theta}\}$ is given by
$$M=\begin{pmatrix}0&0\\
	0&(u''(\theta)+u(\theta))/r
	\end{pmatrix},$$
for every $r>0$ and $\theta\in\mathbb R$. Therefore, the inequality $u''(\theta)+u(\theta)\geq 0$ for every $\theta\in\mathbb R$ is equivalent to $\hes\rho(z)$ being positive semidefinite for every $z\in\mathbb R^2\backslash\{0\}$. The result follows.
\end{proof}
\end{lemma}	 
\begin{proof}
	[Proof of Theorem \ref{th}] Choose any solution $f_0=f_0(\theta)$ of \eqref{eq1} and let the positively homogeneous of degree $1$ function $\rho:\mathbb R^2\to\mathbb R$ be defined as in \eqref{eq3}. By Lemma \ref{lem1}, $\rho$ is convex. Since it is further not linear (because $h$ is not identically zero), Proposition \ref{prop1} states the existence of a supporting linear form $e^*\in(\mathbb R^2)^*$, $(x,y)\mapsto ax+by$. Thus, $u(\theta):=f_0(\theta)-a\cos\theta-b\sin\theta$ satisfies all the required assumptions and completes the proof. 
\end{proof}
\section{Higher eigenvalues}\label{sec4}
The main result of this section is the following
\begin{proposition}\label{prop2}{Let the integer $\omega\geq 3$ be fixed. Then, there exists a $C^\infty$ function $h:\mathbb R/2\pi\mathbb Z\to\mathbb R$ satisfying \eqref{eq2}, with $h(\theta)>0$ for all $\theta\in\mathbb R$ and such that \eqref{eq1} does not have positive solutions.}
\begin{proof}
	We set
	\begin{equation}\label{h}
	h(\theta):=u_*''(\theta)+\omega^2u_*(\theta),\qquad \theta\in\mathbb R,
	\end{equation}
	where $u_*:\mathbb R/2\pi\mathbb Z\to\mathbb R$ will be a suitable $C^\infty$ function satisfying
\begin{itemize}
	\item[{\em (i)}] $u_*''(\theta)+\omega^2 u_*(\theta)>0$\qquad  for all $\theta$,	
	\item[{\em (ii)}] $u_*(0),u_*(3\pi/\omega)<0$.
\end{itemize}
In this way our function $h$ will be $2\pi$-periodic and positive. Moreover, condition \eqref{eq2} follows directly from the definition of $h$. In order to check that  \eqref{eq1} does not have positive solutions we notice that every solution $u=u(\theta)$ has the form
$$u(\theta)=u_*(\theta)+\alpha\sin(\omega\theta)+\beta\cos(\omega\theta),\qquad \theta\in\mathbb R,$$
for some constants $\alpha,\beta\in\mathbb R$. The function $\theta\mapsto\alpha\sin(\omega\theta)+\beta\cos(\omega\theta)$ has a different sign at $\theta=0$ and $\theta=3\pi/\omega$ and therefore {\em (ii)} is not compatible with the existence of solutions $u$ with $u(0),u(3\pi/\omega)>0$.

\medskip

Using a convolution argument with a suitable smooth mollifier $\varphi$ we see that it actually suffices to construct the function  $u_*:\mathbb R/2\pi\mathbb Z\to\mathbb R$ of class $C^1$ and piecewise $C^2$. In this setting, condition {\em (i)} is assumed to hold only on the open intervals where $u_*$ is twice differentiable. The relevant property here is that  equation \eqref{eq1} is linear and therefore if $u=u(\theta)$ is a solution for a given forcing term $h$, then $u\ast\varphi$ is a solution for the regularized forcing term $h\ast\varphi$.

\medskip

With all this in mind we construct $u_*:\mathbb R/2\pi\mathbb Z\to\mathbb R$ as follows:

$$u_*(\theta):=\begin{cases}
	1/2-\cos\omega\theta &\text{if }|\theta|\leq\pi/\omega,\\
	3/2&\text{if }\pi/\omega\leq|\theta|\leq2\pi/\omega,\\
1/2+\cos\omega\theta &\text{if }2\pi/\omega\leq|\theta|\leq\pi,
\end{cases}$$
and extended by periodicity. It is $C^1$ and piecewise $C^2$; moreover, it satisfies {\em (i)-(ii)}, thus concluding the proof. 
\end{proof}
\end{proposition}
 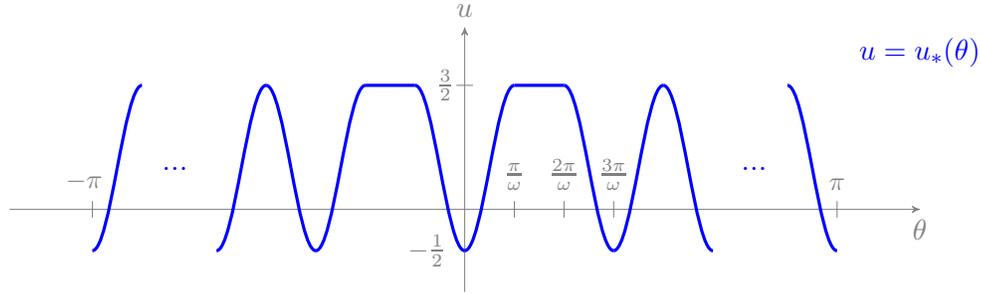
\begin{figure}
	\begin{center}
	\begin{tikzpicture}[scale=1.1]
			\draw[help lines,->] (-5.5,0) -- (5.5,0) node[below] {\small $\theta$};
			\draw[help lines,->] (0,-1) -- (0,2.2) node[above] {\small $u$};
		\draw[blue, very thick](0,-0.5) cos (0.3,0.5);
		\draw[blue, very thick](0.3,0.5) sin (0.6,1.5);
	\draw[blue, very thick](0.6,1.5)--(1.2,1.5);
			\draw[blue, very thick](1.2,1.5) cos (1.5,0.5);
			\draw[blue, very thick](1.5,0.5) sin (1.8,-0.5);
			\draw[blue, very thick](1.8,-0.5) cos (2.1,0.5);
			\draw[blue, very thick](2.1,0.5) sin (2.4,1.5);
				\draw[blue, very thick](2.4,1.5) cos (2.7,0.5);
			\draw[blue, very thick](2.7,0.5) sin (3,-0.5);
				\draw[blue, very thick](3.9,1.5) cos (4.2,0.5);
			\draw[blue, very thick](4.2,0.5) sin (4.5,-0.5);
		\draw[blue, very thick](0,-0.5) cos (-0.3,0.5);
	\draw[blue, very thick](-0.3,0.5) sin (-0.6,1.5);
	\draw[blue, very thick](-0.6,1.5)--(-1.2,1.5);
	\draw[blue, very thick](-1.2,1.5) cos (-1.5,0.5);
	\draw[blue, very thick](-1.5,0.5) sin (-1.8,-0.5);
	\draw[blue, very thick](-1.8,-0.5) cos (-2.1,0.5);
	\draw[blue, very thick](-2.1,0.5) sin (-2.4,1.5);
	\draw[blue, very thick](-2.4,1.5) cos (-2.7,0.5);
	\draw[blue, very thick](-2.7,0.5) sin (-3,-0.5);
		\draw[blue, very thick](-3.9,1.5) cos (-4.2,0.5);
	\draw[blue, very thick](-4.2,0.5) sin (-4.5,-0.5);
		\node[left] at(-0.1,-0.5){\color{gray}\footnotesize $-\frac{1}{2}$};	
		\node[left] at(-0,1.5){\color{gray}\footnotesize $\frac{3}{2}$};	
		\draw[help lines](-0.1,1.5)--(0.1,1.5);
			\draw[help lines](4.5,-0.1)--(4.5,0.1);
				\draw[help lines](-4.5,-0.1)--(-4.5,0.1);
					\node[above] at(4.5,0.1){\color{gray}\footnotesize $\pi$};
					\node[above] at(-4.6,0.1){\color{gray}\footnotesize $-\pi$};
					\node at(3.5,0.5){\color{blue} $...$};	
					\node at(-3.5,0.5){\color{blue} $...$};	
						\node at(5.5,1.9){\color{blue}\small $u=u_*(\theta)$};
							\draw[help lines](0.6,-0.1)--(0.6,0.1);	
								\draw[help lines](1.2,-0.1)--(1.2,0.1);
									\draw[help lines](1.8,-0.1)--(1.8,0.1);	
									\node[above] at(0.6,0.1){\color{gray}\footnotesize $\frac{\pi}{\omega}$};\node[above] at(1.2,0.1){\color{gray}\footnotesize $\frac{2\pi}{\omega}$};			\node[above] at(1.8,0.1){\color{gray}\footnotesize $\frac{3\pi}{\omega}$};	
	\end{tikzpicture}
	
	\end{center}
	\caption{The graph of $u_*$ for $\omega$ odd.}
\end{figure}

We conclude with some final remarks:
\begin{itemize}
\item Concerning the possibility of avoiding subharmonic resonances we notice that the function $u_*$ constructed above satisfies $u_*(-\theta)=u_*(\theta)$ for every $\theta\in\mathbb R$.  In the case $\omega=3$ one further has that $u_*(\pi/2+\theta)=u_*(\pi/2-\theta)$ for every $\theta\in\mathbb R$. Assuming that the mollifier $\varphi$ is chosen even, these symmetries are inherited by $u_*\circ\varphi$, and then by $h:=(u_*\circ\varphi)''+u_*\circ\varphi$, and we deduce that $\int_0^{2\pi}h(\theta)\cos\theta d\theta=\int_0^{2\pi}h(\theta)\sin\theta d\theta=0$. We do not know whether it is possible to construct an improved example for the case $\omega=3$ satisfying also that $\int_0^{2\pi}h(\theta)\cos 2\theta\, d\theta=\int_0^{2\pi}h(\theta)\sin 2\theta\, d\theta=0.$
\item One can check that the trigonometric polynomial $u_*(\theta)=1- 2\cos(2\theta) -\cos(4\theta)$ also satisfies assumptions {\em (i)-(ii)} above in the case $\omega=3$. Thus, for $\omega=3$ the function $h$ defined as in \eqref{h} lies under the framework of Proposition \ref{prop2}.
\item We do not know whether Theorem \ref{th} (or Proposition \ref{prop2}) still holds for $\omega=2$. 
\end{itemize}


\begin{thebibliography}{}
\bibitem{Bre}Brezis, H.,
{\em Analyse Fonctionnelle. Th\'eorie et Applications.} Collection Math\'ematiques Appliqu\'ees pour la Ma\^{i}trise. Masson, Paris, 1983.
\bibitem{deC-Hab}De Coster, C.; Habets, P., {\em Two-point boundary value problems: lower and upper solutions.}  Mathematics in Science and Engineering, 205. Elsevier B. V., Amsterdam, 2006. 
\bibitem{florlev}Florenzano, M.; Le
Van, C., {\em Finite Dimensional Convexity and Optimization.} Studies in Economic Theory; 13. Springer-Verlag, 2001.
\bibitem{Hal} Hale, J.K., {\em Ordinary Differential Equations.} Second edition. Robert E. Krieger Publishing Co., Inc., Huntington, N.Y., 1980.
\bibitem{klee}Klee, V.L.,
{\em Convex sets in linear spaces.}
Duke Math. J. 18 (1951), 443–466.
\bibitem{Min}Minkowski, H., {\em Gesammelte Abhandlungen}, vol. II. B.G. Teubner, Leipzig and Berlin, 1911.
\bibitem{Zhang}Zhang, M.,
{\em Optimal conditions for maximum and antimaximum principles of the periodic solution problem.}
Bound. Value Probl. 2010, Art. ID 410986. 
\end{thebibliography}
\end{document}